\documentclass[journal]{IEEEtran}
\usepackage{}              % list packages between braces
\usepackage[hmargin=2.3cm,vmargin=3.0cm]{geometry}
\usepackage{amsfonts}
\usepackage{amsmath}
\usepackage{amsthm}
\usepackage{bbm}
\usepackage{dsfont}
\usepackage{graphicx}
\usepackage{array}
\usepackage{caption}
\usepackage{subcaption}
\usepackage{mathtools}% type user-defined commands here
\usepackage{xcolor}
\newtheorem{theorem}{Theorem}

\newtheorem{lemma}{Lemma}

\newtheorem{mydef}{Definition}

\newtheorem{corollary}{Corollary}

\newtheorem{remark}{Remark}
\newcommand*{\TitleFont}{%
      \usefont{\encodingdefault}{\rmdefault}{b}{n}%
      \fontsize{18}{26}%
      \selectfont}
\begin{document}
\title{\TitleFont Optimal Energy Procurement from a Strategic Seller with Private Renewable and Conventional Generation}   % type title between braces
\author{Hamidreza Tavafoghi and Demosthenis Teneketzis\\Department of Electrical Engineering and Computer Science\\University of Michigan, Ann Arbor, Michigan, 48109-2122\\Email: \{tavaf,teneket\}$@$umich.edu}         % type author(s) between braces

\date{}    % type date between braces
\maketitle
\vspace{-5pt}
\section*{Abstract}
We consider a mechanism design problem for energy procurement, when there is one buyer and one seller, and the buyer is the mechanism designer.
The seller can generate energy from conventional (deterministic) and renewable (random) plants, and has multi-dimensional private information which determines her production cost. The objective is to maximize the buyer's utility under the constraint that the seller voluntarily participates in the energy procurement process. We show that the optimal mechanism is a menu of contracts (nonlinear pricing) that the buyer offers to the seller, and the seller chooses one based on her private information. 
	
\vspace{5pt}

Keywords: integration of renewable energy, optimal contract, energy procurement, multi-dimensional private information, mechanism design.
\section{Introduction}

The intermittent nature of modern renewable energy resources makes the integration of modern renewable energy generation into the current designed infrastructure for conventional energy generation a challenging problem. Electricity generation from modern renewable energy resources is not predetermined and cannot be treated as conventional generation. Energy generation from wind, solar, and other modern renewable energy resources depends on the weather and is stochastic. This feature results in many technical issues in reliability and stability in generation and transmission, as well as market structure.

Currently, modern renewable energy generators that participate in the regular real-time power market, that is originally designed for conventional generators, are treated as negative loads on the grid, and receive subsidy for each kWh they deliver \cite{east,west}. Due to the increasing share of wind power in energy markets, independent system operators (ISOs) gradually require wind power plants to participate in the day-ahead market, to commit to a fixed amount of generation, and to pay a penalty for each $KWh$ they fail to provide \cite{fink11}.

Along with the current approach, the proposed structure of the smart grid creates additional opportunities to integrate conventional and modern renewable energy resources via flexible loads connected to the grid.

The challenges that arise in the market aspect of the integration of renewable and conventional energy resources motivate the problem formulated and studied in this paper. We follow the current approach for integration of conventional and modern renewable energy resources. This approach requires modern renewable generators to behave as firm power plants that participate in the day-ahead market (in general, a forward market) and commit to producing a predetermined amount of electricity in advance. We want to incorporate new features into our problem formulation, which are mainly motivated by the following observations: (1) Currently, modern renewable generators are paid at a fixed rate, receive subsidy, and do not take any risk, therefore, no strategic behavior is considered. However, as the share of modern renewable generation increases and supportive programs for modern renewable energy decreases, the market becomes more competitive and participants behave strategically. (2) The cost of energy generation varies from 60 to 250 $(\$/MWh)$ due to the different technologies that are available at wind power plants \cite{windcost}. These technologies are the plants' private information. Because modern renewable generators have private production cost and  behave strategically, the current market structure, where renewable generators are paid at a fixed rate, receive subsidies, and all the produced electricity is guaranteed to be procured at the real-time market with no risk, is inefficient.

In this paper we consider a model for integration of conventional and modern renewable energy sources with the following features: (F1) Energy market participants are strategic. (F2) The energy seller has the ability to produce both modern renewable and conventional energy. (F3) Energy producers may have different technologies for renewable and conventional energy generation which are their own private information.

By considering a seller with generation capabilities from both conventional and modern renewable energy resources, we postulate a possible future scenario in which the integration of modern renewable and conventional generation is partly done by sellers and the ISO is not fully responsible for the integration. Our model captures the current approach to the integration of conventional and modern renewable energy resources, that considers sellers with only modern renewable generation capability, if we set the cost of conventional generation to be fixed and equal to the penalty rate for each $KWh$ producers fail to provide.
\vspace{-8pt}
\subsection{Literature Review}
Research on resource allocation and resource management with uncertainty in energy markets has addressed two types of problems:
(1) Those which follow the smart grid approach and focus on the demand side with strategic (\text{e.g.} \cite{fahrioglu}) or non-strategic agents (\textit{e.g.} \cite{lina}). 
(2) Those where the focus is on the supply side and energy providers face uncertainties in their production.

The problem formulated in this paper belongs to the second class where energy providers are strategic and have private information about their own cost and technology. Research on this class of problems has appeared in \cite{chao},\cite{randomselling},\cite{bitar}, and \cite{jain}. The work in \cite{chao} considers a given uncertain demand and investigates a multi-dimensional auction mechanism for the forward reserved market assuming that the participants have no market power and the equilibrium price is not affected by each individual participant's behavior. 
The idea of selling uncertain (random) power to consumers is investigated in \cite{interruppower} and \cite{randomselling}. The problems studied in \cite{bitar} and \cite{jain} consider modern renewable generation and are the most closely related works to the problem we consider here. The work  in \cite{bitar} considers a modern renewable generator with stochastic generation and determines the optimal bidding policy for it in the day-ahead market. The authors of \cite{bitar} assume that price is given and fixed, so there is no role for a buyer, and there is a penalty for production deficiency and over production. The work in \cite{jain} considers a problem where an ISO wants to procure energy from modern renewable generators with the assumption that generation from renewable energy resources is free and each generator has private information about the probability density function of its generation; a VCG-based mechanism is proposed for the optimal energy procurement. The proposed VCG-based mechanism is suboptimal for the problem formulated in \cite{jain} when the probability distribution for the generation cannot be parameterized by only a one-dimensional variable (see \cite{krishna}, chapter 14).

From the economics point of view, the problem we formulate in this paper belongs to the class of screening problems. In economics, the one-dimensional screening problem has been well-studied with both linear and nonlinear utility functions \cite{tilman}. However, the extension to the multi-dimensional screening problem is not straightforward and there is no general solution to it. The authors in \cite{multiscreening} study a general framework for a static multi-dimensional screening problem with linear utility. They discuss two general approaches, the parametric-utility approach and the demand-profile approach. The methodology we use to solve the problem formulated in this paper is similar to the demand-profile approach. We consider a multi-dimensional screening problem with nonlinear utilities. The presence of nonlinearities results in additional complications which are not present in \cite{multiscreening} where the utilities are linear\footnote{When a problem is linear, expectation of any random variable can be replaced by its expected value and reduce the problem to a deterministic one.}. 
\vspace{-7pt}
\subsection{Contribution}
The contribution of this paper is two-fold. First, we introduce a model that captures key features of the current approach for the integration of modern renewable energy sources into the grid, and postulate a possible broader role for sellers in the future. We consider a strategic seller that has the capability to produce energy from modern renewable and conventional resources. To our knowledge, this is the first model that considers simultaneously both types of electricity generation. As we discuss below, considering both modern renewable and conventional generation with a general production cost results in a multi-dimensional mechanism design problem which is conceptually different from the one-dimensional screening problem that arises when there is only one type of energy resource with simple production cost. The model proposed in this paper captures the scenarios investigated in \cite{jain} and \cite{bitar} as special cases; in \cite{bitar} and \cite{jain} the seller owns only modern renewable generators with free generation and is penalized for production mismatch. 

Second, we determine the optimal mechanism for energy procurement from a strategic seller with multi-dimensional private information satisfying both interim and ex-post voluntary participation of the seller. Because of the random nature of renewable energy generation, interim voluntary participation of the seller does not necessarily imply ex-post voluntary participation. To our knowledge, our results present the first optimal mechanism for a strategic seller with conventional and renewable generation, and multi-dimensional private information which also guarantees the seller's ex-post voluntary participation. We show that the current linear pricing for modern renewable energy generation is not efficient and that the optimal pricing method is a nonlinear scheme. 

\subsection{Organization}
The paper is organized as follows: In Section 2, we introduce the  model and formulate the energy procurement problem. In section 3, we present an outline of our approach and the key ideas toward the solution of the problem formulated in section 2, and state the main result on the optimal mechanism for energy procurement. We illustrate the result by an example in section 4. We discuss the nature of our results in section 5. We extend our results to the energy procurement problem without full commitment for the seller, and propose an optimal contract with arbitrary risk allocation between the buyer and the seller in section 6. We conclude in section 7.

\section{Model Specification and Problem Formulation}
\subsection{Model Specification}
A buyer wishes to make an agreement to buy energy from an energy seller\footnote{From now on, we refer to the buyer as ``he'' and to the seller as ``she''.}.
The seller has the ability to produce energy from conventional (deterministic) generators or from renewable (random) generators that she owns.

Let $q$ be the amount of energy the buyer buys, and $t$ be his payment to the seller. We proceed to formulate the energy procurement problem by making the following assumptions.
\vspace{5pt}

\textbf{(A1)}
The buyer is risk-neutral and his total utility is given by $\mathcal{V}(q)-t$, where $\mathcal{V}(q)$ is the utility that he gets from receiving $q$ amount of energy, and $\mathcal{V}(0)=0$. $\mathcal{V}(\cdot)$ is the buyer's private information\footnote{We assume that the buyer either has an elastic demand, or needs to meet some fixed demand and has an outside option to buy energy if he cannot buy it from the seller (which is the result of his interaction with other players in the market).}.
\vspace{5pt}

\textbf{(A2)} The seller's production cost is given by $C(q,w,x)$, where $x\hspace{-2pt}\in\hspace{-2pt}\chi\hspace{-2pt}\subseteq\hspace{-2pt}\mathbb{R}^n$ is the seller's type (technology and cost) and $w$ denotes the realization of a random variable $W$, \textit{e.g.}  weather. $C(q,w,x)$ is convex and increasing in $q$. The start-up cost $C(0,w,x)$ does not depend on the weather $w$ and is given by $x_1$, \textit{i.e.} $C(0,w,x)\hspace{-2pt}=\hspace{-2pt}C(0,x)\hspace{-2pt}=\hspace{-2pt}x_1$.

\vspace{5pt}

\textbf{(A3)} The probability distribution function of $W$, \textit{i.e.} weather forecast, is common knowledge between the buyer and the seller and is given by $F_{W}(w)$. 

\textbf{(A4)}
The seller is risk-neutral and her utility is given by her total expected revenue
\begin{equation}
\mathbb{E}_W\left\{t-C(q,W,x)\right\}.
\end{equation}

\vspace{5pt}

\textbf{(A5)}
Define $c(q,x)\hspace{-2pt}=\hspace{-2pt}\frac{\partial \mathbb{E}_W\left\{C(q,W,x)\right\}}{\partial q}$ as the expected marginal cost for the seller's type $x$. Without loss of generality, there exists $m$, $1\hspace{-2pt}<\hspace{-2pt} \hspace{-2pt}m\hspace{-2pt}\leq n$, such that $c(q,x)$ is increasing in $x_i$ for $1\hspace{-2pt}\leq\hspace{-2pt} i\hspace{-2pt}\leq\hspace{-2pt} m$, and decreasing in $x_i$ for $m\hspace{-2pt}<\hspace{-2pt}i\hspace{-2pt}\leq\hspace{-2pt} n$. Moreover, there is an $\underline{x}\hspace{-2pt}\in\hspace{-2pt}\chi$ (the seller's worst type) such that $\underline{x}_i\hspace{-2pt}\leq\hspace{-2pt} x_i$ and $\underline{x_j}\hspace{-2pt}\geq\hspace{-2pt} x_j$ for all $x\hspace{-2pt}\in\hspace{-2pt}\chi$, $1\hspace{-1pt}\leq \hspace{-2pt}i\hspace{-2pt}\leq\hspace{-2pt} m$ and $m\hspace{-2pt}<\hspace{-2pt}j\hspace{-2pt}\leq \hspace{-2pt}n$. 
\vspace{5pt}

\textbf{(A6)} The seller's type $x$ is her own private information, the set $\chi$ is common knowledge, and there is a prior probability distribution $f_X$ over $\chi$ which is common knowledge between the buyer and the seller.
\vspace{5pt}

\textbf{(A7)}
Both the buyer and the seller are strategic and perfectly rational, and this is common knowledge.
\vspace{5pt}

\textbf{(A8)}
The buyer has all the bargaining power; thus, he can design the mechanism/set of rules that determines the agreement for energy procurement quantity $q$, and payment $t$\footnote{The buyer is either an ISO or a representative agent for aggregate demand. In the first case, it is realistic to assume that the ISO has all the bargaining power since he is the designer and the regulator of the energy market. In the later case, usually there is no competition on the demand side, but different sellers compete to win a contract with the demand side. Therefore, it is reasonable to assume that in a non-cooperative setting, the demand side has all the bargaining power.}.

\vspace{5pt}

\textbf{(A9)}
After the buyer announces the mechanism for energy procurement and the seller accepts it, both the buyer and the seller are fully committed to following the rules of the mechanism.

As a consequence of assumption (A8) on the buyer's bargaining power and the fact that the seller's utility does not directly depend on the buyer's private information, the solution of the problem formulated in this paper does not depend on whether the buyer's utility $\mathcal{V}(\cdot)$ is private information or common knowledge\footnote{This becomes more clear by looking at the main result given by Theorem 1.}.

Note that in the one-dimensional screening problem, the cost of production induces a complete order among the seller's types, which is crucial to the solution to the optimal mechanism design problem. However, in multi-dimensional screening problems, the expected cost of production induces, in general, only a partial order among the seller's types.
\begin{mydef}
We say the seller's type $x$ is better (resp. worse) than the seller's type $\hat{x}$ if $\mathbb{E}_W\left\{C(q,W,x)\right\}\hspace{-2pt}\leq\hspace{-2pt} \mathbb{E}_W\left\{C(q,W,\hat{x})\right\}$ for all $q\hspace{-2pt}\geq\hspace{-2pt}0$ (resp. $\mathbb{E}_W\left\{C(q,W,x)\right\}\hspace{-2pt}\geq\hspace{-2pt} \mathbb{E}_W\left\{C(q,W,\hat{x})\right\}$) with strict inequality for some $q$.
\end{mydef}

From (A5), the seller's type $x$ is better than the seller's type $\hat{x}$ if and only if $x_i\hspace{-2pt}\leq \hspace{-2pt}\hat{x}_i$ for $1\hspace{-2pt}\leq\hspace{-2pt} i\hspace{-2pt}\leq\hspace{-2pt} m$, and $x_i\hspace{-2pt}\geq\hspace{-2pt} \hat{x}_i$ for $m\hspace{-2pt}<\hspace{-2pt} i\hspace{-2pt}\leq\hspace{-2pt} n$ with strict inequality for some $i$. The following example illustrates assumption (A2)-(A5) and Definition 1.

\textbf{A simple case.} Consider a seller with a wind turbine and a gas generator. The generation from the wind turbine is free and given by $\gamma w^3$, where $\gamma$ is the turbine's technology and $w$ is the realized weather. The gas generator has a fixed marginal cost $\theta_c$. We assume that there is a fixed cost $c_0$ which includes the start-up cost for both plants and the capital cost for the seller. Therefore, the seller's type has $n=3$ dimensions. The generation cost for the seller is given by
\begin{eqnarray}
C(q,w,x)=c_0+\theta_c\max\left\{q-\gamma w^3,0\right\}
\end{eqnarray}
The seller's type $x=(c_0,\theta_c,\gamma)$ is better than the seller's type $\hat{x}=(\hat{c}_0,\hat{\theta}_c,\hat{\gamma})$ if and only if $c_0\leq\hat{c}_0$, $\theta_c\leq \hat{\theta}_c$, and $\gamma\geq \hat{\gamma}$ with one of the above inequalities being strict.

\vspace{-3pt}
\subsection{Problem Formulation}
Let $(\mathcal{M},h)$ be the mechanism/game form (see \cite{mascolell}, Ch. 23) for energy procurement designed by the buyer. In this game form, $\mathcal{M}$ describes the message/strategy space and $h$ determines the outcome function; $h:\mathcal{M}\rightarrow\mathbf{R}_+\times\mathbf{R}$ is such that for every message/action $m\in\mathcal{M}$ it specifies the amount $q$ of procured energy and the payment $t$ made to the seller, \text{i.e.} $h(m)=(q(m),t(m))=(q,t)$.

The objective is to determine a mechanism $(\mathcal{M},h)$ so as to
\begin{equation}
\underset{\left(\mathcal{M},(q,t)\right)}{\textnormal{maximize}}\quad \mathbb{E}_{X,W}\left\{\mathcal{V}(q)-t\right\}
\label{P1}
\end{equation}
under assumptions (A1)-(A9) and the constraint that the seller is willing to voluntarily participate in the energy procurement process. The willingness of the seller to voluntarily participate in the mechanism for energy procurement is called voluntary participation (VP) (or individual rationality) and is written as
\begin{eqnarray}
\text{VP:}\quad t(m^*)-\mathbb{E}_W\left\{C(q(m^*),W,x)\right\}\geq0,\quad\forall x\in \chi
\end{eqnarray}
where $m^*\in\mathcal{M}$ is a Bayesian Nash equilibrium (BNE) of the game induced by the mechanism $(\mathcal{M},h)$. That is, at equilibrium the seller has a non-negative payoff.

We call the above problem \textbf{(P1)}.

\section{Outline of the Approach \& Results}

We prove that the optimal energy procurement mechanism is a pricing scheme that the buyer offers to the seller and the seller chooses a production quantity based on her type. We characterize the optimal energy procurement mechanism by the following theorem, which reduces the original functional maximization problem (P1) to a set of equivalent point-wise maximization problems. 

\begin{theorem}
The optimal mechanism $(q,t)$ for the buyer is a menu of contracts (nonlinear pricing) given by
\begin{eqnarray}
\hspace{-20pt}&p(q)=&\textnormal{arg}\max_{\hat{p}}\left\{P\left[x\in\chi|
{\hat{p}}\geq c(q,x)\right]\left(\mathcal{V}'(q)-\hat{p}\right)\right\}\label{opt-p}\hspace{-3pt},\\
\hspace{-20pt}&t(q)=&\int_{0}^q{p(l)dl}+C(0,\underline{x}),\label{optt-gen}\\
\hspace{-20pt}&q(x)=&\textnormal{arg}\max_{l\in\mathbf{R}_+}\mathbb{E}\left\{t\left(l\right)-\mathbb{E}_W\left\{C(l,x,w)\right\}\right\}
\end{eqnarray}
where $\mathcal{V}'(q):=\frac{d\mathcal{V}(q)}{dq}$.
\end{theorem}

The proof of theorem 1 proceeds is several steps. Below we present these steps and the key ideas behind each step. The detailed proof of all results (theorems and lemmas) appearing below can be found in the appendix.

\vspace{5pt}

\textbf{Step 1.} By invoking the revelation principle \cite{dasgupta-revelation}, we restrict attention, without loss of optimality, to direct revelation mechanisms that are incentive compatible (defined below) and individually rational. 

\begin{mydef}
A direct revelation mechanism is defined by functions $q:\chi\rightarrow\mathbf{R}_+$ and $t:\chi\rightarrow\mathbf{R}_+$,
and works as follows:
First, the buyer announces functions $q$ and $t$. Second, the seller declares some $x'\in\chi$ as her report for her technology and cost. Third, the seller is paid $t(x')$ to deliver $q(x')$ amount of energy.
\end{mydef}

The seller is strategic and may lie and misreport her private information $x$, \textit{i.e.} we do not necessarily have $x'=x$ unless it is to the seller's interest to report truthfully. We call incentive compatibility (IC) the requirement for truthful reporting, and define the following problem that is equivalent to (P1).
\vspace{5pt}

\textbf{Problem P2:} Determine functions $q:\chi\rightarrow\mathbf{R}_+$ and $t:\chi\rightarrow\mathbf{R}_+$ so as to
\vspace{-5pt}
\begin{eqnarray}
&&{\underset{(q,t)}{\textnormal{maximize}}}\quad \mathbb{E}_{x,W}\left\{\mathcal{V}(q(x))-t(x)\right\}\\
&&\textit{subject to}\nonumber\\&&\hspace{-25pt}IC: x\hspace{-2pt}=\hspace{-2pt}\textnormal{arg}\max_{x'}\mathbb{E}_W\hspace{-2pt}\left[t(x')\hspace{-1pt}-\hspace{-1pt}C(q(x'),w,x)\right],\forall x\hspace{-2pt}\in\hspace{-2pt}\chi\\
&&\hspace{-25pt}VP: \mathbb{E}_W\left[t(x)-C(q(x),W,x)\right]\geq0,\;\forall x\hspace{-2pt}\in\hspace{-2pt}\chi.
\end{eqnarray}

\vspace{5pt}

\textbf{Step 2.} We utilize the partial order among the seller's different types to rank the seller's utility for her different types, and reduce the VP constraint for all the seller's types to the VP constraint only for the seller's worst type. 

\begin{lemma}
For a given mechanism $(q,t)$, a better type of the seller gets a higher utility. That is, let $U(x):=\mathbb{E}_W\left\{t(q(x))-C(q(x),W,x)\right\}$ denote the 
expected profit of the seller with type $x$. Then,
\begin{enumerate}
\item $\frac{\partial U}{\partial x_i}\leq 0, 1\leq i\leq m$,
\item $\frac{\partial U}{\partial x_i}\geq 0, m< i\leq n$.
\end{enumerate}
\label{ut-order}
\end{lemma}

A direct consequence of Lemma \ref{ut-order}, is that the seller's worst type $\underline{x}$ receives the minimum utility among all the seller's types. 

\begin{corollary}
The voluntary participation constraint is only binding for the worst type $\underline{x}$, that is the general VP constraint (7) can be reduced to
\begin{eqnarray}
U(\underline{x}):=\mathbb{E}\left\{t(q(\underline{x}))-C(q(\underline{x}),w,\underline{x})\right\}=0.
\end{eqnarray}
\end{corollary}
 
\vspace{5pt}

\textbf{Step 3.} We show, via Lemma 2 below, that without loss of optimality, we can restrict attention to a set of functions $t(\cdot)$ that depend only on the amount of energy $q$. That is, the optimal mechanism is a pricing scheme.

\begin{lemma}
For any pair of functions $(q,t)$ that satisfies the IC constraint, we can rewrite $t(x')$ as $t\left(q(x')\right)$.\label{pricing}
\end{lemma}

\textbf{Step 4.} As a consequence of Lemma \ref{pricing}, we determine an optimal mechanism sequentially. First, we determine the optimal payment function $t(\cdot)$, then the optimal energy procurement function $q(\cdot)$. For any function $t(\cdot)$, we determine, for each type $x$ of the seller, the optimal quantity $q^*(x)$ that she wishes to produce as

\vspace{-15pt}
\begin{eqnarray}
&&q^*(x)=\text{arg}\max_{l}\mathbb{E}_W\left\{t\left(l\right)- C(l,w,x)\right\}\label{agent-q}.
\end{eqnarray}
\vspace{-15pt}

Incentive compatibility then requires that the seller must tell the truth to achieve this optimal value, and cannot do better by lying, \emph{i.e.} $q(x)=q^*(x)$ for all $x\in\chi$. For any function $t(\cdot)$, this last equality can be taken as the definition for the associated function $q(\cdot)$. Thus, we eliminate the IC constraint by defining $q(\cdot):=q^*(\cdot)$ and reduce the problem of designing the optimal direct revelation mechanism $(q,t)$ to an equivalent problem where we determine only the optimal payment function $t(\cdot)$ subject to the voluntary participation constraint for the worst type.

\vspace{5pt}

\textbf{Step 5.} To solve this new equivalent problem, we write the buyer's expected utility as the integration of his marginal expected utility, and express the marginal expected utility in terms of the marginal price $p(q):=\frac{\partial t(q)}{\partial q}$ and the minimum payment $t(0)$ (which along with $p(\cdot)$ uniquely determines the payment function $t(\cdot)$). Specifically, in the appendix we show that
%\begin{eqnarray}
%\begin{array}{c}\hspace{0pt}\mathbb{E}_{X}\hspace{-2pt}\left[\mathcal{V}(q^*(X))\right]\hspace{-2pt}\\\hspace{-4pt}-\mathbb{E}_{X}\hspace{-2pt}\left[t(q^*(X))\right]\hspace{-3pt}\end{array}
%\hspace{-7pt}=\hspace{-10pt}
%\begin{array}{c}\hspace{-15pt}\int_{0}^{\infty}\hspace{-3pt}{P\left(x\in\chi | q^*(x)\hspace{-2pt}\geq\hspace{-2pt} l\right)\mathcal{V}'(l)dl}\\-t(0)\hspace{-2pt}-\hspace{-3pt}\int_{0}^{\infty}\hspace{-4pt}{P\hspace{-1pt}\left(x\hspace{-2pt}\in\hspace{-2pt}\chi | q^*\hspace{-1pt}(x)\hspace{-2pt}\geq\hspace{-2pt} l\right)\hspace{-1pt}p(l)dl},\end{array}\hspace{-10pt}\label{buyer-max1}
%\end{eqnarray}
\begin{eqnarray}
&\hspace{-90pt}\mathbb{E}_{X}\hspace{-2pt}\left[\mathcal{V}(q^*(X))\right]\hspace{-2pt}-\hspace{-2pt}\mathbb{E}_{X}\hspace{-2pt}\left[t(q^*(X))\right]\hspace{-3pt}=\nonumber\\&\hspace{15pt}\int_{0}^{\infty}\hspace{-3pt}{P\left(x\in\chi | q^*(x)\hspace{-2pt}\geq\hspace{-2pt} l\right)\mathcal{V}'(l)dl}\nonumber\\&\hspace{40pt}-t(0)-\hspace{-3pt}\int_{0}^{\infty}\hspace{-3pt}{P\left(x\hspace{-1pt}\in\hspace{-1pt}\chi | q^*(x)\hspace{-2pt}\geq\hspace{-2pt} l\right)p(l)dl}\label{buyer-max1},
\hspace{-1pt}
\end{eqnarray}
where $\mathcal{V}'\hspace{-1pt}(q)\hspace{-3pt}:= \hspace{-3pt}\frac{d \mathcal{V}(q)}{dq}$.
That is, the buyer's total expected utility is obtained by integrating his marginal utility at quantity $l$,  times the probability that the seller's production exceeds $l$,  and subtracting the minimum payment $t(0)$.
We show in the appendix that the seller's optimal decision $q^*(x)$ depends only on the marginal price $p(q)$. Thus, 
we can write the probability associated with the seller's decision as 
\begin{eqnarray}
P\left(x\in\chi | q^*(x)\geq l\right)=P\left[x\in\chi|
p(l)\geq c(l,x)\right]\label{seller-p}.
\end{eqnarray}
That is, the seller is willing to produce the marginal quantity at $l$ if the resulting expected marginal profit is positive, \textit{i.e.} marginal price $p(l)$ exceeds marginal expected cost of generation $c(l,x)$\footnote{This relies on quasi-concavity of the seller's optimal decision program. This is a standard assumption in the literature, e.g. see \cite{multiscreening} and \cite{nonlinear}. Basically, it gives the seller the freedom to decide for each marginal unit of production independently. Therefore, the continuity of the resulting generation quantity must be checked posterior to the design of the optimal contract for each type of the seller.}. Using (\ref{buyer-max1}) and (\ref{seller-p}), we define the following problem that is equivalent to (P2) and is in terms of the marginal price $p(q)$ and the minimum payment $t(0)$.

\vspace{5pt}
\textbf{Problem P3:}
\begin{eqnarray}
&\hspace{-40pt}{\underset{p(\cdot),t(0)}{\textnormal{max}}}&\hspace{-30pt}
\int_{0}^{\infty}\hspace{-9pt}{P\hspace{-2pt}\left[x\hspace{-2pt}\in\hspace{-2pt}\chi|
p(l)\hspace{-2pt}\geq\hspace{-2pt} c(l,x)\right]\hspace{-2pt}\left(\mathcal{V}'(l)\hspace{-2pt}-\hspace{-2pt}p(l)\right)\hspace{-1pt}dl}\hspace{-2pt}-\hspace{-2pt}t(0)\\
&\hspace{-10pt}\textit{subject to}&\nonumber\\&\hspace{-46pt} \text{VP:}&\hspace{-37pt}\mathbb{E}_W\hspace{-2pt}\left\{t(0)\hspace{-1pt}+\hspace{-2pt}\int_0^{q(\underline{x})}\hspace{-10pt}p(r)dr-C(q^*(\underline{x}),w,\underline{x})\right\}\hspace{-4pt}\geq\hspace{-2pt}0.\label{p4-vp}
\end{eqnarray}

\textbf{Step 6.} We provide a ranking for the seller's optimal decision $q^*(x)$ based on the partial order among the seller's types.

\begin{lemma}
For a given mechanism specified by $(t(\cdot),q(\cdot))$, a better type of the seller produces more. That is, the optimal quantity $q^*(x)$  that the seller with true type $x$ wishes to produce satisfies the following properties:
\begin{description}
\item[a)] $\frac{\partial q^*(x)}{\partial x_i}\leq 0,1\leq i\leq m$,
\item[b)] $\frac{\partial q^*(x)}{\partial x_i}\geq 0, m<x\leq n$.
\end{description}
\end{lemma}

In the appendix we show that a consequence of corollary 1 and Lemma 3 is the following result.
\begin{corollary}
The VP constraint is satisfied if $t(0)\hspace{-3pt}=\hspace{-3pt}C(0,\hspace{-1pt}\underline{x})$ and the lowest seller's type payment is equal to her expected production cost, \textit{i.e.} $t(q(\underline{x}))\hspace{-2pt}=\hspace{-2pt}\mathbb{E}_W\hspace{-2pt}\left\{C(q(\underline{x}),W,\underline{x})\right\}$.
\end{corollary} 

Based on corollary 2, we define a problem (P4) that is equivalent to (P3) and is only in terms of the marginal price $p(l)$ and the constraint that the payment the  seller's lowest type receives is equal to her cost of production. 

\vspace{5pt}

\textbf{Problem (P4)}
\begin{eqnarray}
&\hspace{-50pt}{\underset{p(\cdot)}{\textnormal{max}}}&\hspace{-30pt} \int_{0}^{\infty}{\hspace{-5pt}P\left[x\in\chi|
p(l)\geq c(l,x)\right]\left(\mathcal{V}'(l)-p(l)\right)dl}\\
&\hspace{-15pt}\text{subject to}&\nonumber\\  
&&\hspace{-50pt}\text{VP:}\hspace{4pt}
C(0,\underline{x})+\int_0^{q(\underline{x})}\hspace{-5pt}{p(l)dl}=\mathbb{E}_W\left[C(q(\underline{x}),W,x)\right]\label{P4-VP}.
\label{p5-con}
\end{eqnarray}

\textbf{Step 7.} We consider a relaxed version of (P4) without the VP constraint (\ref{P4-VP}). The unconstrained problem can be solved point-wise at each quantity $l$ to determine the optimal $p(l)$ as
\vspace{-3pt}
\begin{eqnarray}
\hspace{-20pt}&p(l)=&\hspace{-5pt}\textnormal{arg}\max_{\hat{p}}\left\{P\left[x\hspace{-1pt}\in\hspace{-1pt}\chi|
{\hat{p}}\geq c(q,x)\right]\hspace{-1pt}\left(\mathcal{V}'(q)-\hat{p}\right)\hspace{-1pt}\right\}\label{margp-gen}\hspace{-3pt}
\end{eqnarray}
\vspace{-1pt}
which is the same as (\ref{opt-p}). We show in the appendix that the solution to the unconstrained problem automatically satisfies the VP constraint (\ref{P4-VP}), therefore, the solution to the unconstrained problem determines the optimal marginal price $p(\cdot)$ for the original problem.
Using $p(\cdot)$ and the fact that $t(0)=C(0,\underline{x})$, we determine the optimal function $t(\cdot)$. Using $t(\cdot)$ we find the seller's best response function $q^*(\cdot)$. By incentive compatibility $q^*(\cdot)=q(\cdot)$, and this completely determines the optimal direct revelation mechanism $(q(\cdot),t(\cdot))$ described by Theorem 1.

In essence, Theorem 1 states that at each quantity $l$, the optimal marginal price $p(l)$ is chosen so as to maximize the expected total marginal utility at $l$, which is given by the total marginal utility $\left(\mathcal{V}'(l)-p(l)\right)$ times the probability that the seller generates at least $l$.

\vspace{5pt}
\begin{remark}
In a setup with startup cost for the seller, it might not be optimal for the buyer to require all the seller's types to voluntarily participate in the energy procurement process, since the minimum payment $t(0)$ depends on the production cost of the seller's worst type. In such cases, it might be optimal for the buyer to exclude some ``less efficient'' types of the seller from the contract, select an admissible set of seller's types, and then design the optimal contract for this admissible set of the seller's types\footnote{To find the optimal admissible set, the optimal contract can be computed for different potential admissible sets. Then, the resulting utilities can be compared to find the best admissible set.}. Note that, this is not the case for setups without startup cost. In such setups, if it is not optimal for some type $x$  to be included in the optimal contract, it is equivalent to have $q(x)=0$ for the optimal contract that considers all types of the seller.
\end{remark}

\begin{remark}
In problem (P1), we assume that there exists a seller's worst type which has the highest cost at any quantity among all the seller's types, and we reduce the VP constraint for all the seller's type to only the VP constraint for this worst type. As a result, we pin down the optimal payment function by setting $t(0)=C(0,\underline{x})$ to ensure the voluntary participation of the worst type, which consequently implies the voluntary participation for all the seller's types. In absence of the assumption on the existence of the seller's worst type, the argument used to reduce the VP constraint is not valid anymore and we cannot pin down the payment function and specify $t(0)$ a priori. Assuming that all types of the seller participate in the contract, their decision on the optimal quantity $q^*$ only depends on the marginal price $p(q)$, and therefore, the optimal marginal price $p(q)$ given by (\ref{margp-gen}) is still valid without the assumption on the existence of the worst type. To pin down the payment function $t(\cdot)$, we find the minimum payment $t(0)$ a posteriori so that all types of the seller voluntarily participate. That is,
\vspace{-3pt}
\begin{eqnarray}
t(0)=\max_{x\in\chi}\left[\mathbb{E}_W\left\{C(q(x),w,x)\right\}-\int_0^{q(x)}p(q)dq\right]
\end{eqnarray}
  \vspace{-1pt}
where the optimal decision of type $x$ is given by
  \vspace{-1pt}
\begin{eqnarray}
q(x)=\textnormal{arg}\max_q \left[\int_0^q p(q)-\mathbb{E}_W\left\{C(q,w,x)\right\}\right].
\end{eqnarray}
\end{remark}

\section{Example}
Consider a seller with a conventional plant and a wind turbine.
The wind turbine's output power curve $g(w)$ is as in figure 1. 
We assume that for wind speeds between $v_{ci}$ and $v_r$ the energy generation is given by $\gamma w^3$, where $\gamma$ captures the technology and size of the turbine. Energy generation remains constant for wind speeds between $v_r$ and $v_{co}$, and is zero otherwise.

%We characterize the output power curve $g(w)$ with $4$ variables: the cut-in speed $v_{ci}$ at which turbine starts to rotates, the rated speed $v_r$ the nominal speed that the turbine is designed for, and the cut-out speed $v_{co}$. We assume that the generation for wind speed between $v_{ci}$ and $v_r$ is given by $\gamma w^3$ where $\gamma$ captures the technology and the size of the turbine and generation remains constant for wind speed between $v_r$ and $v_{co}$, and is zero otherwise. 
\vspace{-5pt}
\begin{figure}[h!]
\centering
\includegraphics[width=0.30\textwidth, height=0.10\textheight]{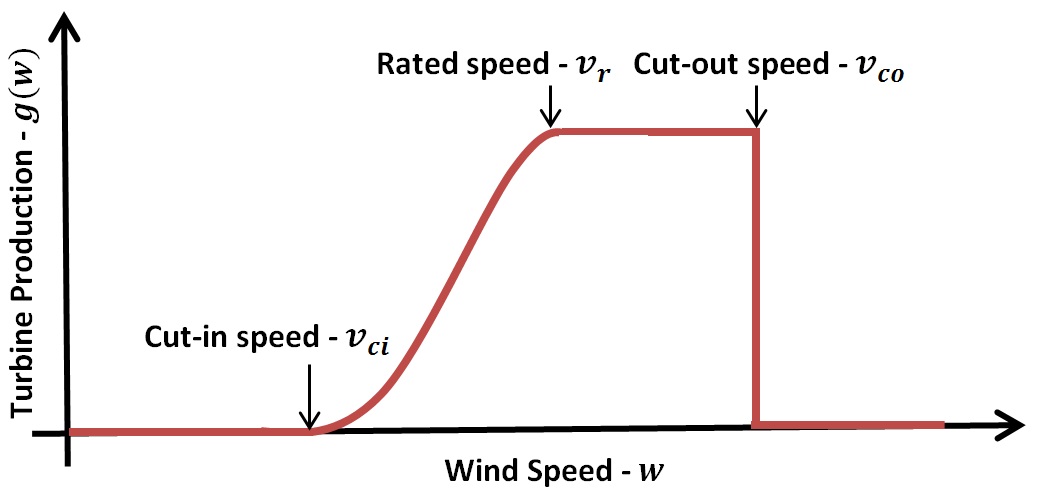}
  \caption{Example - the wind turbine generation curve}
  \label{curve}
\end{figure}

We assume that there is a fixed marginal operational cost $\theta_w$ for the wind turbine and a fixed marginal operational cost $\theta_c$ for the conventional plant. 
There is a no-production cost $c_0$ that captures the start-up cost for both plants and the capital cost for the seller. Therefore, the seller's type $x\hspace{-2pt}=\hspace{-2pt}(c_0,\theta_w,\theta_c,v_{ci},v_{r},\hspace{-1pt}v_{co},\hspace{-1pt}\gamma)$ is 7-dimensional and her total cost is given by
\begin{eqnarray}
C(q,w,\hspace{-1pt}x)\hspace{-2pt}=\hspace{-2pt}c_0\hspace{-2pt}+\hspace{-2pt}\theta_w\hspace{-2pt}\min\hspace{-2pt}\left\{\hspace{-1pt}q,g(w)\hspace{-1pt}\right\}\hspace{-3pt}+\hspace{-1pt}\theta_c\hspace{-2pt}\max\hspace{-2pt}\left\{\hspace{-1pt}q\hspace{-2pt}-\hspace{-2pt}g(w),\hspace{-1pt}0\hspace{-1pt}\right\}\hspace{-3pt},\hspace{-4pt}
\end{eqnarray} 
where $g(w)$ is as in figure \ref{curve}. The wind profile is a class $k\hspace{-2pt}=\hspace{-2pt}3$ Weibull distribution with average speed $5m/s$. 

We only consider 6 types for the seller here:\vspace{-6pt}
\begin{eqnarray}
&a&\hspace{-10pt}=\hspace{-2pt}(c_0\hspace{-4pt}=\hspace{-4pt}4,\theta_w\hspace{-4pt}=\hspace{-4pt}0.2,\theta_c\hspace{-4pt}=\hspace{-4pt}1.2,v_{ci}\hspace{-4pt}=\hspace{-4pt}3,v_{r}\hspace{-4pt}=\hspace{-4pt}13,v_{co}\hspace{-4pt}=\hspace{-4pt}20,\gamma\hspace{-4pt}=\hspace{-4pt}1),\nonumber\\
&b&\hspace{-10pt}=\hspace{-2pt}(c_0\hspace{-4pt}=\hspace{-4pt}4,\theta_w\hspace{-4pt}=\hspace{-4pt}0.2,\theta_c\hspace{-4pt}=\hspace{-4pt}1.2,v_{ci}\hspace{-4pt}=\hspace{-4pt}3,v_{r}\hspace{-4pt}=\hspace{-4pt}13,v_{co}\hspace{-4pt}=\hspace{-4pt}20,\gamma\hspace{-4pt}=\hspace{-4pt}2),\nonumber\\
&c&\hspace{-10pt}=\hspace{-2pt}(c_0\hspace{-4pt}=\hspace{-4pt}5,\theta_w\hspace{-4pt}=\hspace{-4pt}0.1,\theta_c\hspace{-4pt}=\hspace{-4pt}1.2,v_{ci}\hspace{-4pt}=\hspace{-4pt}3,v_{r}\hspace{-4pt}=\hspace{-4pt}13,v_{co}\hspace{-4pt}=\hspace{-4pt}20,\gamma\hspace{-4pt}=\hspace{-4pt}1),\nonumber\\ &d&\hspace{-10pt}=\hspace{-2pt}(c_0\hspace{-4pt}=\hspace{-4pt}5,\theta_w\hspace{-4pt}=\hspace{-4pt}0.2,\theta_c\hspace{-4pt}=\hspace{-4pt}1.0,v_{ci}\hspace{-4pt}=\hspace{-4pt}1,v_{r}\hspace{-4pt}=\hspace{-4pt}17,v_{co}\hspace{-4pt}=\hspace{-4pt}28,\gamma\hspace{-4pt}=\hspace{-4pt}2),\nonumber\\ &e&\hspace{-10pt}=\hspace{-2pt}(c_0\hspace{-4pt}=\hspace{-4pt}6,\theta_w\hspace{-4pt}=\hspace{-4pt}0.1,\theta_c\hspace{-4pt}=\hspace{-4pt}1.0,v_{ci}\hspace{-4pt}=\hspace{-4pt}1,v_{r}\hspace{-4pt}=\hspace{-4pt}17,v_{co}\hspace{-4pt}=\hspace{-4pt}28,\gamma\hspace{-4pt}=\hspace{-4pt}1),\nonumber\\ &f&\hspace{-10pt}=\hspace{-2pt}(c_0\hspace{-4pt}=\hspace{-4pt}6,\theta_w\hspace{-4pt}=\hspace{-4pt}0.1,\theta_c\hspace{-4pt}=\hspace{-4pt}1.0,v_{ci}\hspace{-4pt}=\hspace{-4pt}1,v_{r}\hspace{-4pt}=\hspace{-4pt}13,v_{co}\hspace{-4pt}=\hspace{-4pt}28,\gamma\hspace{-4pt}=\hspace{-4pt}2),\nonumber
\end{eqnarray} where the cost unit is $\$1000$ and the energy unit is $MWh$, and there is no worst type. The optimal contract from Theorem 1 is depicted in Figure \ref{example2}. It is a nonlinear pricing scheme. The marginal price varies between $0.33$ and $0.45$ $\$/KWh$. The variation in the marginal price is of the same order as the variation in the expected marginal production cost across the seller's different types.
\begin{figure}[h!]
\centering
\includegraphics[width=0.51\textwidth]{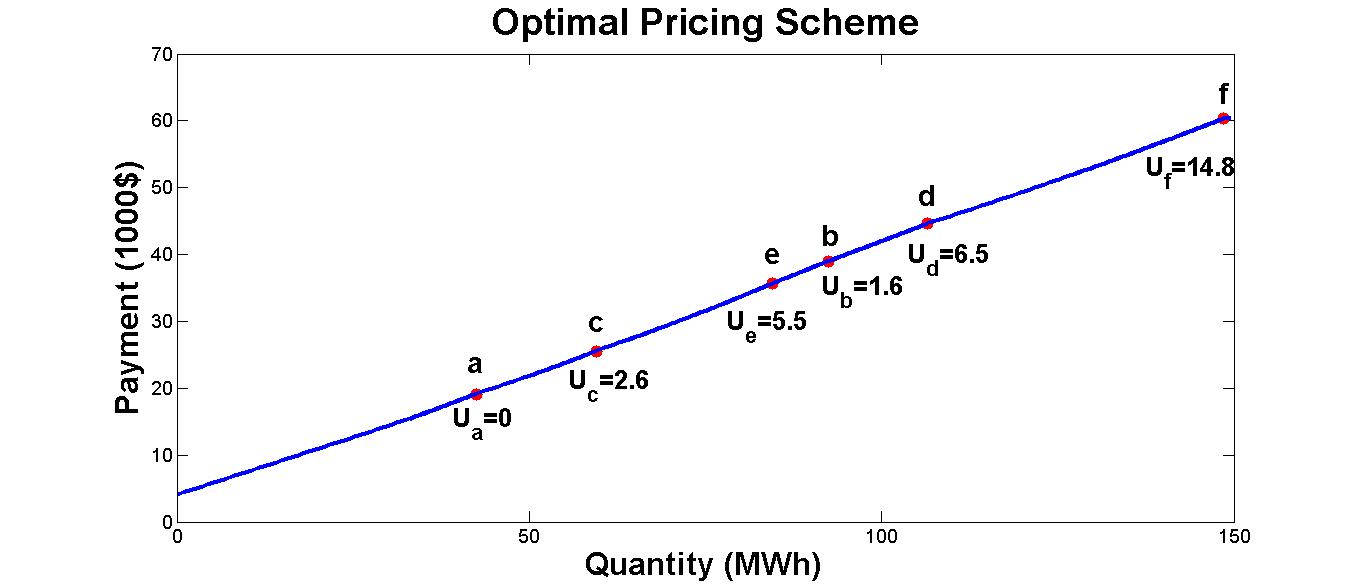}
  \caption{Example - the optimal pricing scheme}
    \vspace{-5pt}
  \label{example2}
\end{figure}
Since there is no complete order among the different seller's types, we can not compare the seller's types based on the expected revenue or amount of production prior to the design of the mechanism.
%Since there is no complete oder among different types of the seller we cannot a-prior to design of the contract rank different types of the seller based on their final revenue or amount of production. 
However, wherever we have a partial order and can rank a subset of types, we can utilize Lemmas 1 and 3 to predict a ranking a priori. For instance, the seller with type $(b)$ is better than the one with type $(a)$, and we can say a priori that the former has a higher production and revenue. However, we cannot order types $(b)$ and $(c)$. 
For the setup of our example, according to the optimal contract, type $(c)$ gets a higher expected revenue than type $(b)$ but produces a lower amount of energy than type $(b)$.
%As it can be seen, in this specific setup and the resulting optimal contract, $(c)$ has a higher revenue but lower production than $(b)$.

  \vspace{-4pt}
\section{Discussion}

The  optimal mechanism/contract for the energy procurement problem formulated in this paper is a nonlinear pricing scheme. The nonlinearity is due to three factors. First, the buyer's utility function $\mathcal{V}(q)$ is not linear in the quantity $q$. Second, for each type of the seller, the cost of production is a nonlinear function of the amount of produced energy. Third, the seller has private information about her technology and cost (seller's type).

The uncertainty about the production from the non-conventional plant makes the total expected production cost function nonlinear even with a fixed marginal cost of production for conventional and renewable plants. The buyer has to pay information rent (monetary incentive) to the seller to incentivize her to reveal her true type. Therefore, the payment the buyer makes to the seller includes the cost of production the seller incurs plus the information rent, which varies with the seller's type; the better the seller's type, the higher is the information rent.

The optimal contract/mechanism discovered in this paper can be implemented indirectly (without reporting the seller's private information) as follows: the buyer offers the seller a menu of contracts (nonlinear pricing scheme); the seller chooses one of these contracts based on her type, and there is no need for an implausible and unnecessary information exchange stage between the seller and the buyer.

The multi-dimensionality of the seller's private information could be due to the different types of energy generators that she owns, or because of a complex cost function with more than one parameter (even with one type of energy generator). The solution approach presented in this paper can be used to solve problems with multi-dimensional private information with similar structure.

The optimal contract defined by Theorem 1 also induces some incentives for investment in infrastructure and technology development. From Lemma 1, the seller with the higher type has a higher utility. Therefore, there is an incentive for the seller to improve her technology and decrease her cost of generation.

It is well-known that in the presence of private information and strategic behavior, in general, there exists no mechanism/contract that is (1) individually rational, (2) incentive compatible, and (3) efficient (Pareto-optimal generally) \cite{rosenthal}. In the optimal contract/mechanism given by Theorem 1, the allocation for the seller's  different types is not ex-post efficient (Pareto-optimal) except for the seller's worst type who gets zero utility.
%\footnote{This is called \textit{``distortion at the top''} in the contact theory literature\cite{bolton}.}.

\section{Further Considerations}
\subsection{Commitment and Ex-post Voluntary Participation}

%In the design of the optimal mechanism/contract, 
The voluntary participation constraint imposed in problem (P1) is interim. That is, the expected profit with respect to the weather for each type of the seller must be non-negative. By assumption (A9), once the seller agrees on the contract (this agreement takes place before the realization of the weather) she is fully committed to following the agreement, even if the realized profit is negative (because of the realization of the weather)\footnote{Since the seller's reserve utility is zero by not participating (outside option), we can always think of the seller walks away from the agreement for these negative profit realizations and not follow the mechanism rules.}. 
%. To relax the full commitment assumption we must impose an ex-post voluntary participation constraint in the specification of problem (P1)
%By ex-post voluntary participation constraint we mean that for all realizations of the weather, the seller's realized profit must be non-negative. We replace the interim VP constraint (7) by the following,
To ensure that the seller's realized profit is non-negative for all weather realizations, we impose an ex-post voluntary participation constraint. We replace the interim VP constraint (7) by 
  \vspace{-1pt}
\begin{eqnarray}
&\hspace{-20pt}\text{Ex-post VP:}&\hspace{-5pt}t(m^*)-C(q,w,x)\geq 0,\forall w,\;\forall x\hspace{-2pt}\in\hspace{-2pt}\chi.
\label{P1'-VP}
\end{eqnarray}
  \vspace{-2pt}
To obtain an ex-post voluntary participation constraint, we modify the payment function of the mechanism given by Theorem 1 as follows:
  \vspace{-1pt}
\begin{eqnarray}
\tilde{t}(q,w)\hspace{-2pt}=\hspace{-2pt}t(q)-t(q(\underline{x}))+C(q(\underline{x}),w,\underline{x}).
\end{eqnarray}
  \vspace{-1pt}
We have $\mathbb{E}_W\hspace{-3pt}\left\{\tilde{t}(q,w)\hspace{-1pt}\right\}\hspace{-3pt}=\hspace{-3pt}t(q)$, and therefore, the seller always chooses the same quantity $q$ under the modified payment function $\tilde{t}(\cdot)$ as under the original payment function $t(q)$ given by (\ref{optt-gen}). Note that for the seller's worst type, we have $\tilde{t}(q(\underline{x}),\hspace{-1pt}w)\hspace{-2pt}=\hspace{-2pt}C(q(\underline{x}),\hspace{-1pt}w,\underline{x})$, and therefore, the total utility of the seller's worst type is zero for all realizations of $W$. Since all other types of the seller are better off than the worst type  under $\tilde{t}(\cdot)$ (all types have the choice to produce the same quantity as the worst type)
%at least by choosing the same quantity as the worst type)
, the ex-post VP constraint is satisfied for all of the seller's types.

  \vspace{-5pt}
\subsection{Risk Allocation}

In the optimal mechanism/contract menu presented by Theorem 1, the buyer faces no uncertainty, and he is guaranteed to receive quantity $q(x)$, and all the risk associated with the realization of the weather is taken by the seller. We wish to modify the mechanism to reallocate the above-mentioned risk between the buyer and the seller. To do so, we modify the payment function so that the risk is reallocated between the buyer and the seller. Consider the following modified payment function with $\alpha\hspace{-2pt}\in\hspace{-2pt}[0,1]$,
  \vspace{-1pt}
\begin{eqnarray}
&\hat{t}(x,w)=&t(q(x))+\alpha\left[C(q(x),w,x)\right.\nonumber\\&&\left.-\mathbb{E}_W\left\{C(q(x),w,x)\right\}\right]\label{t-risk}.
\end{eqnarray}
From (\ref{t-risk}) it follows that $\mathbb{E}_W\hspace{-3pt}
\left\{\hat{t}(x,w)\right\} \hspace{-3pt}=\hspace{-2pt}t(q(x))$. Therefore, the strategic behavior of the seller does not change and the seller chooses the same quantity under the modified payment function $\hat{t}(\cdot)$ as under the original payment function $t(q)$ given by (\ref{optt-gen}). Note that for $\alpha=0$ we have the same payment as $t(q)$. For $\alpha=1$, the seller is completely insured against any risk and all the risk is taken by the buyer. The parameter $\alpha$ determines the allocation of the risk between the buyer and the seller; the buyer undertakes $\alpha$ and the seller undertakes $(1-\alpha)$ share of the risk.

\section{Conclusion} 

We proposed a model for energy procurement that captures the current approach for integration of modern renewable and conventional energy resources and takes into account the fact that renewable energy producers behave strategically and may have different production technologies. We analyzed an arbitration between a strategic energy buyer and a strategic energy seller who has the ability to generate energy from a conventional plant and a modern renewable plant and has private information about her production technology and cost. We showed that the optimal contract/mechanism for energy procurement is a nonlinear pricing scheme.
The originally proposed mechanism guarantees interim voluntary participation of the seller. By modifying the seller's payment so as to be weather-dependent we achieved ex-post voluntary participation of the seller. We also presented an alternative payment to the seller that divides the risk (due to the uncertainty in the weather) between the buyer and the seller.

\section{Acknowledgment}

The authors would like to thank Galina Schwartz for her helpful comments on the paper. This work was supported in part by NSF Grant CNS-1238962.

\addcontentsline{toc}{section}{References}
\bibliographystyle{plain}
\bibliography{bib}
\section*{\LARGE{Appendix}}   % type title 
\vspace{10pt}
\section*{Details and Proofs of the Results}
\vspace{5pt}
Consider the following problem (P1) formulated in the paper.
\textbf{Problem (P1):} 
\begin{eqnarray}
&\hspace{-25pt}{\underset{(\mathcal{M},(q,t))}{\textnormal{maximize}}}& \mathbb{E}_{X,W}\left\{\mathcal{V}(q)-t\right\}
\label{P1}\\\nonumber\vspace{-6pt}\\
&\textit{subject to}\nonumber&\\
&&\hspace{-65pt}\text{VP:}\; t(m^*)\hspace{-2pt}-\hspace{-2pt}\mathbb{E}_W\hspace{-3pt}\left\{C(q(m^*),W,x)\right\}\hspace{-2pt}\geq\hspace{-2pt}0,\forall x\hspace{-2pt}\in \hspace{-3pt}\chi.
\end{eqnarray}
where $\left\{\mathcal{M},(q,t).q:\mathcal{M}\hspace{-1pt}\rightarrow\hspace{-1pt} \mathbb{R}_+,t:\mathcal{M}\hspace{-1pt}\rightarrow\hspace{-1pt} \mathbb{R}_+\right\}$  denotes the mechanism to be designed, and the notation in (1) and (2) is the same as in the paper.

The main result in the paper is given by Theorem 1 stated below.
\begin{theorem}
The optimal mechanism $(q,t)$ for the buyer is a menu of contracts (nonlinear pricing) given by
\begin{eqnarray}
\hspace{-20pt}&p(q)=&\textnormal{arg}\max_{\hat{p}}\left\{P\left[x\in\chi|
{\hat{p}}\geq c(q,x)\right]\left(\mathcal{V}'(q)-\hat{p}\right)\right\}\label{margp-gen}\hspace{-3pt},\\
\hspace{-20pt}&t(q)=&\int_{0}^q{p(l)dl}+C(0,\underline{x}),\label{optt-gen}\\
\hspace{-20pt}&q(x)=&\textnormal{arg}\max_{l\in\mathbf{R}_+}\mathbb{E}\left\{t\left(l\right)-\mathbb{E}_W\left\{C(l,x,w)\right\}\right\}.
\end{eqnarray}
\end{theorem}

\vspace{5pt}
In this note, we provide all the details of the proof of Theorem 1 that were left out of the presentation in the paper due to lack of space. We follow the same steps as in the paper.

\begin{proof}

%We want to design an optimal energy procurement mechanism $(\mathcal{M},(q,t))$ for the buyer, where $q:\mathcal{M}\to\mathbb{R}_+$ and $t:\mathcal{M}\to\mathbb{R}_+$,  so as to maximize the buyer's expected utility; we call this problem (P1).

We proceed to solve (P1) and prove theorem 1 by the following steps. 

\vspace{5pt}

\textbf{Step 1.} By invoking the revelation principle \cite{dasgupta-revelation}, we restrict attention, without loss of optimality, to direct revelation mechanisms that are incentive compatible (defined below) and individually rational (individual rationality is equivalent to voluntary participation). 

\begin{mydef}
A direct revelation mechanism is defined by functions $q:\chi\rightarrow\mathbf{R}_+$ and $t:\chi\rightarrow\mathbf{R}_+$.
and works as follows: 
\begin{itemize}
\item First, the buyer announces functions $q$ and $t$.
\item Second, the seller declares some $x'\in\chi$ as her report for her type.
\item Third, the seller is paid $t(x')$ to deliver $q(x')$ amount of energy.
\end{itemize}
\end{mydef}

Note that the seller is strategic and may lie and misreport her private information $x$, \textit{i.e.} we do not necessarily have $x'=x$.

\vspace{5pt}

\textbf{The revelation principle:}
For any BNE $\tilde{m}^*$ of the game induced by an arbitrary mechanism $(\tilde{\mathcal{M}},(\tilde{q},\tilde{t}))$, there exists an equivalent direct revelation mechanism $\left(\chi,(q,t)\right)$, in which truthful reporting is a BNE of the game induced by $\left(\chi,(q,t)\right)$, and the players' allocation and payment associated with the truth-telling equilibrium are identical to those associated with the BNE $\tilde{m}^*$ of the original mechanism $(\tilde{\mathcal{M}},(\tilde{q},\tilde{t}))$.

\vspace{5pt}
In essence, by invoking the revelation principle we eliminate the problem of finding the optimal message space $\mathcal{M}$ by restricting attention to direct revelation mechanisms ($\mathcal{M}\hspace{-4pt}:=\hspace{-4pt}\chi$), and impose a new set of incentive compatibility (IC) constraints. As a result, we can solve the following problem (P2) to find an optimal direct revelation mechanism.
\vspace{5pt}

\textbf{Problem P2:} Determine functions $q:\chi\rightarrow\mathbf{R}_+$ and $t:\chi\rightarrow\mathbf{R}_+$ so as to
\begin{eqnarray}
&&\hspace{-25pt}{\underset{(q,t)}{\textnormal{maximize}}}\quad \mathbb{E}_{X,W}\left\{\mathcal{V}(q(X))-t(X)\right\}\\\nonumber\\
&&\hspace{-15pt}\textit{subject to:}\nonumber\\&&\hspace{-25pt}IC\hspace{-3pt}: x\hspace{-2pt}=\hspace{-2pt}\textnormal{arg}\max_{x'}\mathbb{E}_W\hspace{-2pt}\left[t(x')\hspace{-2pt}-\hspace{-2pt}C(q(x'),W,x)\right],\forall x\hspace{-2pt}\in\hspace{-2pt}\chi,\\
&&\hspace{-25pt}VP\hspace{-3pt}: \mathbb{E}_W\left[t(x)-C(q(x),W,x)\right]\geq0,\;\forall x\in\chi.
\end{eqnarray}

\vspace{5pt}
\textbf{Step 2.} We utilize the partial order among the seller's different types to order the seller's resulting utility for her different types and reduce the VP constraint for all of the seller's types to the VP constraint only for the seller's worst type. 

\begin{lemma}
For a given incentive compatible mechanism $(q,t)$, a better type of the seller gets a higher utility. That is, let
$U(x):=\mathbb{E}_W\left\{t(q(x))-C(q(x),W,x)\right\}$ denote the 
expected profit of the seller with type $x$. Then,
\begin{enumerate}
\item $\frac{\partial U}{\partial x_i}\leq 0, 1\leq i\leq m$,
\item $\frac{\partial U}{\partial x_i}\geq 0, m< i\leq n$.
\end{enumerate}
\label{ut-order}
\end{lemma}

\begin{proof}[Proof of lemma 1]
The given mechanism $(q,t)$ is incentive compatible, so we can rewrite $U(x)$ as
\begin{eqnarray}
U(x)=\max_{x'}{\mathbb{E}_W\left\{t(q(x'))-C(q(x'),W,x)\right\}}\label{Lemma1-U}
\end{eqnarray}
By applying the envelope theorem \cite{envelope} on (\ref{Lemma1-U}), we get 
\begin{eqnarray}
\frac{\partial U}{\partial x_i}=-\left.\frac{\partial \mathbb{E}_W\left\{C(q(x'),W,x)\right\}}{x_i}\right|_{x'=x}.
\end{eqnarray}
The above equation along with assumption (A5) on the marginal expected cost, gives
\begin{eqnarray}
&&\frac{\partial U}{\partial x_i}\leq 0,1\leq i \leq m\\
&&\frac{\partial U}{\partial x_i}\geq 0,m< i \leq n
\end{eqnarray}

\end{proof}
A direct consequence of Lemma 1 is the following.
\begin{corollary}
The voluntary participation constraint, is satisfied if and only if it is satisfied for the worst type $\underline{x}$, that is the general VP constraint (8) can be reduced to
\begin{eqnarray}
U(\underline{x})=\mathbb{E}_W\left\{t(q(\underline{x}))-C(q(\underline{x}),W,\underline{x})\right\}\geq0.
\end{eqnarray}
\end{corollary}

\textbf{Step 3.} We show, via lemma 2 below, that without loss of optimality, we can restrict attention to a set of functions $t(\cdot)$ that depend only on the amount of delivered energy $q$. That is, the optimal mechanism is a pricing scheme.

\begin{lemma}
For any given pair of functions $(q,t)$ that satisfies the IC constraint, we can rewrite $t(x')$ as $t\left(q(x')\right)$.\label{pricing}
\end{lemma}

\begin{proof}[Proof of lemma 2]
The proof is by contradiction. Assume that there exist $x,x'\in\chi$ such that $q(x)=q(x')$ but $t(x')>t(x)$. Then a seller with type $x$ is always better off by reporting $x'$ instead of her true type $x$, which contradicts the IC constraint.
\end{proof}
\vspace{5pt}
\textbf{Step 4.} 
As a consequence of lemma \ref{pricing}, we determine an optimal mechanism sequentially. First, we determine the optimal payment function $t(\cdot)$, then the optimal energy procurement function $q(\cdot)$.

For any given pair of functions $(q,t)$, a seller with true type $x$ solves the following maximization problem to find her optimal report $x'^*$,
\begin{eqnarray}
x'^*=\textnormal{arg}\max_{x'}\mathbb{E}_W\left\{t\left(q(x')\right)-C(q(x'),W,x)\right\}.
\label{agent-gen}
\end{eqnarray}

Considering the solution of (\ref{agent-gen}) for every type $x\hspace{-2pt}\in\hspace{-2pt}\chi$,  we can form a function $q^*:\chi\longrightarrow \mathbf{R}_+$ defined by
\begin{eqnarray}
q^*(x):=\textnormal{arg}\max_{l}\mathbb{E}_W\left\{t\left(l\right)- C(l,W,x)\right\}\label{agent-q}.
\end{eqnarray}
For a given direct revelation mechanism, the function $q^*(\cdot)$ defines the optimal quantity of energy procured from each type of the seller. Incentive compatibility then requires that the seller must tell the truth to achieve this optimal value, and cannot do better by lying, \emph{i.e.} 
\begin{eqnarray}
q(x)=q^*(x)\quad \forall x\in\chi.
\label{agent-rep}
\end{eqnarray}

The function $q^*(\cdot)$ is induced only by $t(\cdot)$ through (\ref{agent-q}), and therefore, (\ref{agent-rep}) can be taken as the definition for the associated function $q(\cdot)$ that along with $t(\cdot)$ satisfies the set of IC constraints. Thus, we can eliminate the IC constraint by defining $q(\cdot):=q^*(\cdot)$ and reduce the problem of designing the optimal direct revelation mechanism $(q,t)$ to an equivalent problem (P2') where we determine only the optimal payment function $t(\cdot)$ subject to the voluntary participation constraint for the worst type.
\vspace{5pt}

\textbf{Problem P2':} Determine function $t:\mathbf{R}_+\rightarrow\mathbf{R}_+$ so as to
\begin{eqnarray}
&&\hspace{-10pt}{\underset{t(\cdot)}{\textnormal{maximize}}}
\quad \mathbb{E}_{X}\left\{\mathcal{V}(q^*(X))-t(q^*(X))\right\}\label{buyer-obj-p2}\\
&&\hspace{-0pt}\text{subject to}\nonumber\\
&&\hspace{-10pt}\text{VP:}\; \mathbb{E}_W\hspace{-2pt}\left[t(q(\underline{x}))\hspace{-2pt}-\hspace{-2pt}C(q^*(\underline{x}),W,\underline{x})\right]\hspace{-2pt}\geq\hspace{-2pt}0,
\label{ind-3}
\end{eqnarray}
where $q^*$ is given by (\ref{agent-q}).

\vspace{5pt}

\textbf{Step 5.} To solve problem (P2'), we show that the optimal decision of the seller for amount of power $q^*$ depends only on the marginal price $p(q):=\frac{\partial t(q)}{\partial q}$ and express the buyer's expected utility in term of the marginal price.

Consider the buyer's objective (\ref{buyer-obj-p2}). For any function $t(\cdot)$, we can determine from (\ref{agent-q}) the cumulative distribution function for $q^*$, called $F_{q^*}$. Consequently, we can rewrite the buyer's objective as

\begin{eqnarray}
&&\mathbb{E}_{q^*}\left[\mathcal{V}(q^*)-t(q^*)\right]=\int_{0}^{\infty}{\left(\mathcal{V}(l)-t(l)\right)dF_{q^*}(l)}\nonumber\\
&&=\left.\left(F_{q^*}(l)-1\right)\left(\mathcal{V}(l)-t(l)\right)\right|_{0}^{\infty}\nonumber\\&&+
\int_{0}^{\infty}{\left(1-F_{q^*}(l)\right)\frac{d\left(\mathcal{V}(l)-t(l)\right)}{d l}dl}.
\label{maxobj-1}
\end{eqnarray}

We have
\begin{eqnarray}
\left.\left(F_{q^*}(l)-1\right)\left(\mathcal{V}(l)-t(l)\right)\right|_{0}^{\infty}=-t(0)
\label{t0}
\end{eqnarray}
because $\mathcal{V}(0)=0$ by assumption, and $\left(F_{q^*}(\infty)-1\right)=0$.

Because of (\ref{t0}), we can rewrite (\ref{maxobj-1}) as
\begin{eqnarray}
&\mathbb{E}_{q^*}\left[\mathcal{V}(q^*)-t(q^*)\right]&\hspace{-5pt}=\hspace{-3pt}\int_{0}^{\infty}\hspace{-3pt}{P\left(q^*\geq l\right)\left(\mathcal{V}'(l)-p(l)\right)dl}
\nonumber\\&&-t(0)
\label{buyer-max1}
\end{eqnarray}

where $\mathcal{V}'(l)=\frac{d \mathcal{V}(l)}{d l}$. 

\vspace{5pt}

We can rewrite $P\left(q^*\geq l\right)$ as
\begin{eqnarray}
&&\hspace{-20pt}P\hspace{-1pt}\left(q^*\hspace{-2pt}\geq\hspace{-2pt} l\right)\hspace{-3pt}=\hspace{-3pt}P[x\hspace{-2pt}\in\hspace{-2pt}\chi|\text{arg}\hspace{-1pt}\max_l\mathbb{E}_W\hspace{-3pt}\left\{t(q)\hspace{-2pt}-\hspace{-2pt}C(q(x),\hspace{-1pt}W,\hspace{-1pt}x)\right\}\hspace{-3pt}\geq\hspace{-2pt} l].\nonumber\\&&
\label{q-demand1}
\end{eqnarray}

We implicitly assume that the seller's problem given by (\ref{agent-q}) is continuous and quasi-concave\footnote{This is a standard assumption in literature, e.g. see \cite{multiscreening} and \cite{nonlinear}. Basically, it can be seen as a situation where the seller can decide for each marginal unit of production independently. Therefore, in general, there is no guarantee that the seller's independent decisions about each marginal unit of production results in a continuous and plausible total production quantity $q$. Therefore, the continuity of the result must be checked posteriori for each type of the seller.}, so that from the first order optimality condition for (\ref{agent-q}) we obtain
\begin{eqnarray}
p(q^*(x))=\left.\frac{\partial \mathbb{E}_W\left\{C(l,W,x)\right\}}{\partial l}\right|_{q^*(x)}.
\label{foc}
\end{eqnarray}
Therefore, each type of the seller wishes to produce more than quantity $l$ if and only if the marginal price $p(q)$ that she is paid at $l$ is higher than the expected marginal cost of production $c(l,x)$ that she incurs at $l$. Consequently, combining (\ref{q-demand1}) and (\ref{foc}) we obtain
\begin{eqnarray}
P\left(q^*\geq l\right)=P\left[x\hspace{-2pt}\in\hspace{-2pt}\chi|
p(l)\geq\frac{\partial \mathbb{E}_W\left\{C(l,W,x)\right\}}{\partial l}\right]
\label{q-demand2}.
\end{eqnarray}

Substituting (\ref{q-demand2}) in (\ref{buyer-max1}), we obtain the following alternative expression for the buyer's objective

\begin{eqnarray}
&\hspace{-8pt}\mathbb{E}_{q^*}\left[\mathcal{V}(q^*)\hspace{-2pt}-\hspace{-2pt}t(q^*)\right]\hspace{-2pt}=&\hspace{-12pt}
\int_{0}^{\infty}\hspace{-6pt}{P\hspace{-2pt}\left[x\hspace{-2pt}\in\hspace{-2pt}\chi|
p(l)\hspace{-2pt}\geq\hspace{-2pt}\frac{\partial \mathbb{E}_W\hspace{-3pt}\left\{C(l,\hspace{-1pt}W,\hspace{-1pt}x)\right\}}{\partial l}\right]}\nonumber\\&&\hspace{10pt}\left(\mathcal{V}'(l)\hspace{-2pt}-\hspace{-2pt}p(l)\right)dl\hspace{-3pt}-t(0).
\label{buyer-max2}
\end{eqnarray}

The buyer seeks to determine the marginal price $p(\cdot)$ and $t(0)$ to maximize the right hand side of (\ref{buyer-max2}), subject to the VP constraint for the seller's worst type.
Based on this consideration, we define the following problem (P3) that is equivalent to problem (P2').

\vspace{10pt}

\textbf{Problem P3:}
\begin{eqnarray}
&&\hspace{-20pt}{\underset{p(\cdot),t(0)}{\textnormal{maximize}}}\quad\hspace{-8pt}
 -\hspace{-1pt}t(0)\hspace{-2pt}+\hspace{-4pt}\int_{0}^{\infty}\hspace{-7pt}{P\hspace{-1pt}\left[x\hspace{-2pt}\in\hspace{-2pt}\chi|
p(l)\hspace{-2pt}\geq\hspace{-2pt} \frac{\partial \mathbb{E}_W\hspace{-3pt}\left\{C(l,w,x)\right\}}{\partial l}\right]}\nonumber\\&&\hspace{80pt}\left(\mathcal{V}'(l)-p(l)\right)dl\\
&&\textit{subject to}\quad\nonumber\\&& \hspace{-19pt}\text{VP:}\;\max_{l}{\hspace{-1pt}t(0)\hspace{-2pt}+\hspace{-4pt}\int_0^{q^*\hspace{-1pt}(\underline{x})}\hspace{-11pt}p(r)dr\hspace{-2pt}-\hspace{-1pt}\mathbb{E}_W\hspace{-3pt}\left\{C(q^*\hspace{-1pt}(\underline{x}),W,\underline{x})\hspace{-1pt}\right\}}\hspace{-4pt}\geq\hspace{-2pt}0.
\end{eqnarray}

\textbf{Step 6.} We provide a ranking for the seller's optimal decision $q^*(x)$ based on the partial order among the seller's types via the following lemma.

\begin{lemma}
For a given mechanism specified by $(t(\cdot),q(\cdot))$, a better type of the seller produces more. That is, the optimal quantity $q^*(x)$  that the seller with true type $x$ wishes to produce satisfies the following properties:
\begin{description}
\item[a)] $\frac{\partial q^*(x)}{\partial x_i}\leq 0,1\leq i\leq m$,
\item[b)] $\frac{\partial q^*(x)}{\partial x_i}\geq 0, m<x\leq n$.
\end{description}
\end{lemma}

\begin{proof}[Proof of lemma 3]
Let $x,x'\in\chi$, where $x$ is a better type than $x'$. From IC for seller's type $x$ we have
\begin{eqnarray}
&t(q(x))\hspace{-2pt}-\hspace{-2pt}\mathbb{E}_W\hspace{-3pt}\left\{C(q(x),W,x)\right\}\hspace{-2pt}\nonumber\\&\geq\nonumber\\&\hspace{-2pt} t(q(x'))\hspace{-2pt}-\hspace{-2pt}\mathbb{E}_W\hspace{-3pt}\left\{C(q(x'),W,x)\right\}\label{lemma3-1}
\end{eqnarray}
Similarly from IC for seller's type $x'$ we have
\begin{eqnarray}
&t(q(x'))-\mathbb{E}_W\left\{C(q(x'),W,x')\right\}\nonumber\\&\geq\nonumber\\& t(q(x))-\mathbb{E}_W\left\{C(q(x),W,x')\right\}\label{lemma3-2}
\end{eqnarray}
Subtracting (\ref{lemma3-2}) from (\ref{lemma3-1}), we get
\begin{eqnarray}
&\mathbb{E}_W\left\{C(q(x),W,x')\right\}-\mathbb{E}_W\left\{C(q(x'),W,x')\right\}\nonumber\\&\geq\nonumber\\ &\mathbb{E}_W\left\{C(q(x),W,x)\right\}-\mathbb{E}_W\left\{C(q(x'),W,x)\right\}\label{lemma3-3}
\end{eqnarray}
By assumption (A4), $\frac{d\mathbb{E}_W\hspace{-2pt}\left\{C(q,W,x)\right\}}{dq}\hspace{-2pt}\leq\hspace{-2pt} \frac{d\mathbb{E}_W\hspace{-2pt}\left\{C(q,W,x')\right\}}{dq}$ if $x$ is a better type than $x'$.
%the expected cost $\mathbb{E}_W\left\{C(\cdot,W,\cdot)\right\}$ is decreasing difference (\textit{i.e.} the marginal expected cost decreases as type gets better). 
Therefore, (\ref{lemma3-3}) holds if and only if
\begin{eqnarray}
q(x)\geq q(x').
\end{eqnarray}
\end{proof}

The following result is a consequence of corollary 1 and lemma 3.
\begin{corollary}
The VP constraint is satisfied if $t(0)=C(0,\underline{x})$ and the lowest seller's type is paid exactly equal to her expected production cost, \textit{i.e.} $t(q(\underline{x}))=\mathbb{E}_W\left\{C(q(\underline{x}),W,\underline{x})\right\}$.
\end{corollary}
\begin{proof}[Proof of corollary 2]
Because of corollary 1, the VP constraint implies
\begin{eqnarray}
U(\underline{x})=t(q(\underline{x}))-\mathbb{E}_W\left[C(q^*(\underline{x}),W,x)\right] = 0,
\end{eqnarray} 
which is equivalent to
\begin{eqnarray}
t(0)+\int_0^{q^*(\underline{x})}{p(l)dl}=\mathbb{E}_W\left[C(q^*(\underline{x}),W,x)\right].
\label{VP-red}
\end{eqnarray}

Furthermore, from Lemma 3 it follows that if the worst type wishes to produce more than $q^*(\underline{x})$, then all types produce more than $q^*(\underline{x})$. Therefore,
\begin{eqnarray}
P\left[x\in\chi|
p(l)\geq c(l,x)\right]=1,\;\; \text{for}\; l\leq q^*(\underline{x}).
\label{pforlmin}
\end{eqnarray}

Using (\ref{pforlmin}), we can rewrite the objective function of problem (P3) as,
\begin{eqnarray}
&\hspace{-20pt}&-\left(t(0)+
\int_0^{q^*(\underline{x})}{p(l)dl}\right)+\int_0^{q^*(\underline{x})}{\mathcal{V}'(l)dl}
\nonumber\\&\hspace{-20pt}&+
\int_{q^*(\underline{x})}^{\infty}{P\left[x\in\chi|
p(l)\geq c(l,x)\right]\left(\mathcal{V}'(l)-p(l)\right)dl}\hspace{-1pt}.
\label{obj-red}
\end{eqnarray}

The term $t(0)+\int_0^{q^*(\underline{x})}{p(l)dl}$ appears in both the objective (\ref{obj-red}) and the VP constraint (\ref{VP-red}). Therefore, without loss of optimality, we can assume $t(0)=C(0,\underline{x})$, and set $t(q(\underline{x}))=\mathbb{E}_W\left\{C(q(\underline{x}),W,\underline{x})\right\}$.
\end{proof}

Using corollary 2, we define a problem (P4) that is equivalent to (P3) and is only in terms of the marginal price $p(\cdot)$ and the constraint that the payment the  seller's lowest type receives is equal to her cost of production. 

\vspace{5pt}

\textbf{Problem (P4)}
\begin{eqnarray}
&\hspace{-50pt}{\underset{p(\cdot)}{\textnormal{max}}}&\hspace{-15pt} \int_{0}^{\infty}{\hspace{-5pt}P\left[x\in\chi|
p(l)\geq c(l,x)\right]\left(\mathcal{V}'(l)-p(l)\right)dl}\\
&&\hspace{-25pt}\text{subject to}\nonumber\\&&\hspace{-30pt}\text{VP:}\;
C(0,\underline{x})+\int_0^{q^*(\underline{x})}{\hspace{-4pt}p(l)dl}\hspace{-1pt}=\hspace{-1pt}\mathbb{E}_W\left[C(q^*(\underline{x}),w,x)\right]\label{P4-VP}\hspace{-1pt}.
\label{p5-con}
\end{eqnarray}

\textbf{Step 7.} We consider a relaxed version of (P4) without the VP constraint (\ref{P4-VP}). The unconstrained problem (P4) can be solved by maximizing the integrand $P\left[x\in\chi|
p(l)\geq c(l,x)\right]\left(\mathcal{V}'(l)-p(l)\right)$ point-wise at each quantity $l$. 
The solution of the point-wise maximization problem for the optimal marginal price $p(\cdot)$ is given by,
\begin{eqnarray}
\hspace{-20pt}&p(l)\hspace{-2pt}=&\hspace{-8pt}arg\max_{\hat{p}}\left\{P\left[x\in\chi|
{\hat{p}}\geq c(q,x)\right]\left(\mathcal{V}'(q)-\hat{p}\right)\right\}\label{margp-gen}\hspace{-3pt}.
\end{eqnarray}
Using (\ref{pforlmin}), along with the the fact that the worst type has the highest expected marginal cost, we can simplify (\ref{margp-gen}) for $l\leq q^*(\underline{x})$,
\begin{equation}
p(l)= c(l,\underline{x}),\;\;\text{for}\;l\leq q^*(\underline{x}).
\end{equation}
That is,  for $l\leq q^*(\underline{x})$, the minimum marginal price $p(l)$ that ensures all the seller's type are willing to produce more than $q^*(\underline{x})$ is equal to the marginal expected cost for the seller's worst type $c(l,\underline{x})$.
Therefore, the solution to the unconstrained version of problem (P4) satisfies condition (\ref{p5-con}) of problem (P4), and therefore, (\ref{margp-gen}) (which is the same as (3)) is also the optimal solution of problem (P4).

From Corollary 2 and (\ref{margp-gen}), the optimal payment function (nonlinear pricing) is given by,
\begin{eqnarray}
t(q)=\int_{0}^q{p(l)dl}+C(0,\underline{x})
\end{eqnarray}
which is the same as (4).
From (\ref{agent-q}) we determine the optimal energy procurement function,
\begin{eqnarray}
q(x)=\text{arg}\max_{l}\mathbb{E}\left\{t\left(l\right)-C(l,w,x)\right\} 
\end{eqnarray} 
which us the as (5).
The specification of $t(\cdot)$ and $q(\cdot)$ completes the proof of theorem 1 and the solution to problem (P1).

\end{proof}

\end{document}